%% file: 000_main.tex
\documentclass{amsart}

\input{000_setup}

\title{New Properties of Harmonic Polygons}
\author{Ronaldo Garcia}
\thanks{R. Garcia,  Inst. de Matem\'atica e Estatística,\\ Univ.Federal de Goiás, Brazil. \texttt{ragarcia@ufg.br}}
\author{Dan Reznik}
\thanks{D. Reznik$^*$, Data Science Consulting Ltd., Rio de Janeiro, Brazil. \texttt{dreznik@gmail.com}}
\author{Pedro Roitman}
\thanks{P. Roitman, Departamento de Matemática,\\ Universidade de Brasília, Brazil. \texttt{roitman@mat.unb.br}}

\begin{document}

\maketitle
\vspace{-2em}
\input{005_abstract}

\section{Introduction}
\input{010_introduction}

\section{Two new conservations}
\label{sec:conserved}
\input{020_conserved}

\section{Conserved sums of cotangents}
\label{sec:cotangents}

\input{025_cotangents}


\section{Harmonics and homothetics}
\label{sec:homothetics}
\input{040_homothetics}

\section{Isocurves of Brocard angle}
\label{sec:isobrocs}
\input{050_isobrocs}

\section{Videos}
\label{sec:videos}
\input{090_videos}

\section*{Acknowledgements}
\input{120_ack}

\appendix

\section{Review: Harmonic polygons}
\label{app:review}
\input{230_app_review_harmonic}

\section{Harmonic family: explicit formulas}
\label{app:harm}
\input{220_app_harmonic}





\bibliographystyle{maa}
\bibliography{references,authors_rgk} 
\end{document}

%% file: 000_setup.tex
\usepackage{graphics}
\usepackage{thmtools}
\usepackage{wasysym}

\usepackage{iftex}                      
\iftutex                                
 \defaultfontfeatures {Ligatures=TeX}
 \usepackage{unicode-math}              
 \unimathsetup{math-style=TeX}
 \ifLuaTeX
   \usepackage{lualatex-math}
 \fi
\else                                   
\usepackage[T1]{fontenc}                
\fi
\usepackage{alphabeta}              

\usepackage{amsthm}
\usepackage{amsbsy,amsmath,amssymb,amscd,amsfonts}
\usepackage[pagebackref=false]{hyperref}
\usepackage[nameinlink,capitalize,noabbrev]{cleveref}

\usepackage{graphicx,float,latexsym,color}
\usepackage[font={scriptsize,it}]{caption}
\usepackage{subcaption}
\usepackage[export]{adjustbox}

\usepackage{makecell}

\makeatletter\def\defcommand{\@ifstar\defcommand@S\defcommand@N} \def\defcommand@S#1{\let#1\outer\renewcommand*#1} \def\defcommand@N#1{\let#1\outer\renewcommand#1} \makeatother

\usepackage[dvipsnames]{xcolor}

\newtheorem{theorem}{Theorem}
\newtheorem*{theorem*}{Theorem}

\newtheorem{proposition}{Proposition}
\newtheorem{conjecture}{Conjecture}
\newtheorem{corollary}{Corollary}
\newtheorem{lemma}{Lemma}
\theoremstyle{remark}

\theoremstyle{definition}
\newtheorem{definition}{Definition}

\hypersetup{
    pdftoolbar=true,        
    pdfmenubar=true,        
    pdffitwindow=false,     
    pdfstartview={FitH},    
    colorlinks=true,       
    linkcolor=OliveGreen,          
    citecolor=blue,        
    filecolor=black,      
    urlcolor=red           
}

\usepackage{lineno}

\usepackage{comment}
\usepackage{microtype}
\usepackage{footnote}
\defcommand{\A}{\mathcal{A}}  
\defcommand{\F}{\mathcal{F}}  
\defcommand{\C}{\mathcal{C}}                  
\newcommand{\R}{\mathcal{R}}
\newcommand{\B}{\mathcal{B}}
\renewcommand{\P}{\mathcal{P}}  
\renewcommand{\H}{\mathcal{H}}
\renewcommand{\S}{\mathcal{S}}

\newcommand{\E}{\mathcal{E}}		

\newcommand{\torp}[2]{\texorpdfstring{#1}{#2}}

%% file: 005_abstract.tex
\begin{abstract}
Via simulation, we revisit the Poncelet family of ``harmonic polygons'', much studied in the 2nd half of the XIX century by famous geometers such as Simmons, Tarry, Neuberg, Casey, and others. We review its (inversive and projective) construction, identify some new conservations, and contrast it, via its invariants, to several other recently studied Poncelet families.
\vskip .3cm
\noindent\textbf{Keywords} harmonic polygon, Poncelet, Brocard, invariants, projection, homothetic, inversion, symmetric polynomials.
\vskip .3cm
\noindent \textbf{MSC} {51M04
\and 51N20 \and 51N35\and 68T20}
\end{abstract}

%% file: 010_introduction.tex
Following results by Brocard and Lemoine in the first half of the XIX century, {\em harmonic polygons} were discovered and intensely studied decades later by such geometers as Casey, McCay, Neuberg, Simmons, Tarry, Vigarié, and others, see \cite[Chapter VIII]{simmons1886} for the historical background.

\begin{figure}[H]
\begin{subfigure}{0.49\textwidth}
\includegraphics[trim=0 0 0 15,clip,width=0.95\linewidth,frame]{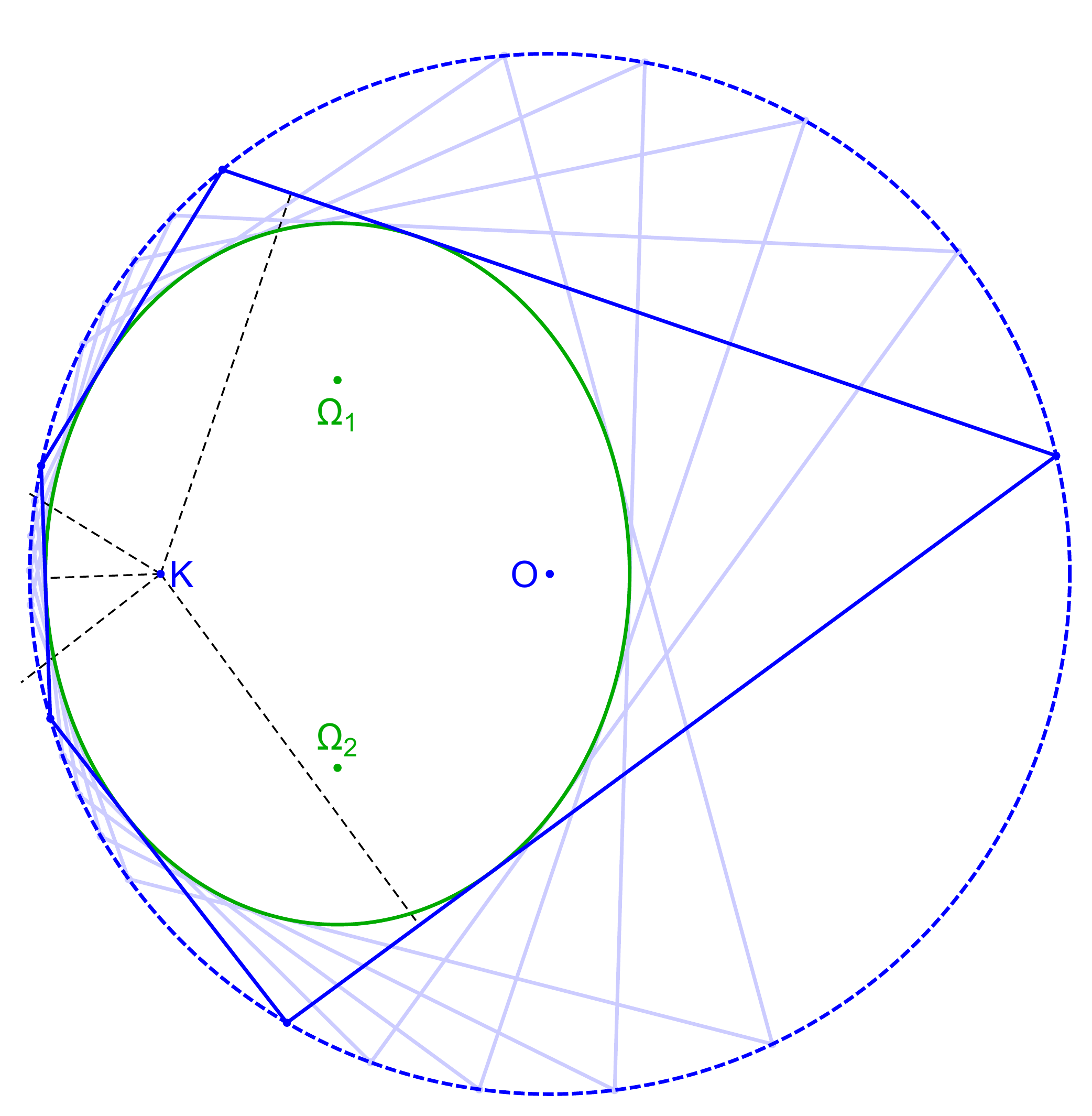}
\end{subfigure}
\begin{subfigure}{0.5\linewidth}
\includegraphics[trim=275 25 250 25,clip,width=.99\textwidth,frame]{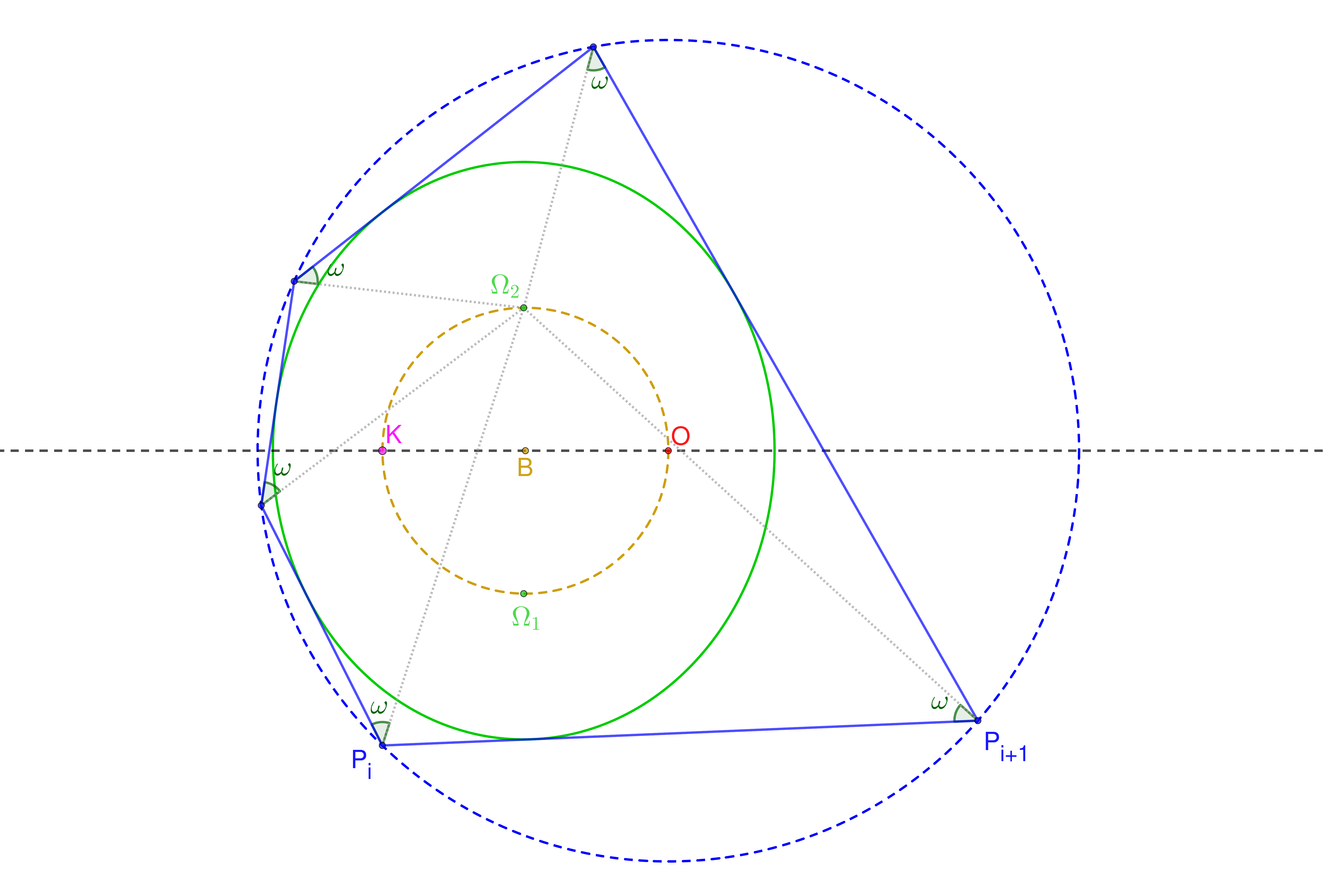}
\end{subfigure}
\caption{\textbf{Left:} The Poncelet harmonic family (blue) is inscribed in a circle centered at $O$ and contains a point $K$ (symmedian) whose distance to the sides is proportional to the sidelengths. The caustic (green) is known as the Brocard inellipse whose foci are the {\em Brocard points} $\Omega_1$ and $\Omega_2$.
\textbf{Right:} said Brocard points are where sides $P_i P_{i+1}$ rotated an angle $\omega$ about $P_i$ (resp. $-\omega$ about $P_{i+1}$) concur. $\omega$ is known as the {\em Brocard angle}. Also shown is the Brocard circle (dashed brown) with diameter $KO$.}
\label{fig:basic}
\end{figure}

Referring to \cref{fig:basic}(left), a polygon $\P$ is harmonic if inscribed in a circle $\C$ and containing an interior point $K$ (known as the {\em symmedian point}) whose distance to each sideline is a fixed proportion of the sidelength.

$\P$ circumscribes a special conic known as the {\em Brocard inellipse}, so named since its foci $\Omega_1,\Omega_2$ are the {\em Brocard points} of $\P$. These are points of concurrence of rotations of each side by a fixed angle known as the {\em Brocard angle} $\omega$, see \cref{fig:basic}(right).

Since $\P$ is interscribed between two real conics, Poncelet's closure theorem applies and a 1d family of such polygons will exist \cite{dragovic11,sergei91}. Amazingly, over the Poncelet family, $\Omega_1,\Omega_2$ (and many other associated objects) remain stationary and $\omega$ remains constant.
A review of harmonic polygons appears in \cref{app:review}.

\subsection*{Main Results} Using a simulation-based approach (mostly with Mathematica \cite{mathematica_v10}), we detected the following phenomena manifested by harmonic polygons which, to the best of our knowledge, had not been yet described.

\subsection{New conservations} the following conservations are proved in \cref{sec:conserved}:
\begin{itemize}
\item The sum of inverse squared sidelengths.
\item The sum of inverse squared radii of {\em Apollonius' circles} which are generalizations of same-named circles in triangles \cite{mw};
\item The sum of powers of internal angle cotangents, as well as all elementary symmetrical functions thereof (except for one).
\end{itemize}

In \cref{tab:conserve} the above conservations are compared side-by-side with others manifested by other Poncelet families studied in  \cite{akopyan2020-invariants,bialy2020-invariants,caliz2020-area-product,galkin2021-affine,reznik2021-fifty-invariants,roitman2021-bicentric}.

\subsection{Relationship to the Poncelet Homothetic family} In \cref{sec:homothetics} we show that a certain polar image of the harmonic family is the so-called ``homothetic family'', i.e., a Poncelet family of $N$-gons interscribed between two homothetic ellipses.

Therefore, and as shown in \cref{fig:homot-polars}, two ``lateral'' harmonic families can be obtained from the homothetic one: these are polar images of the latter with respect to the left (resp. right) focus of their inner ellipse. We show that the harmonic mean of their areas is invariant for $N=3,5$ and conjecture this will hold for all $N$.

\begin{figure}
    \centering
    \includegraphics[width=\textwidth]{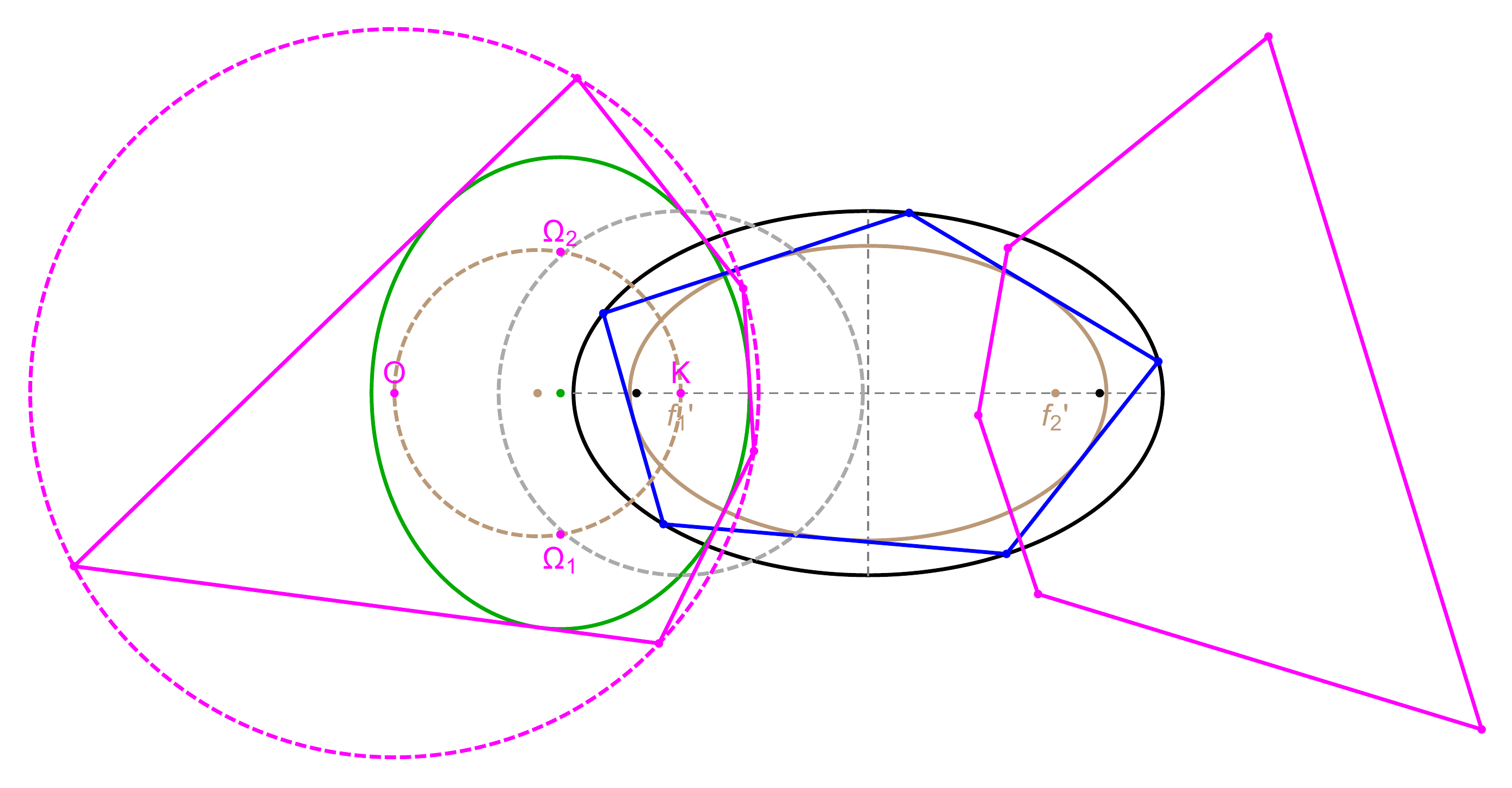}
    \caption{The Poncelet homothetic family (blue), interscribed between two homothetic ellipses (black and brown) is the polar image of a harmonic family (left, magenta) with respect to its symmedian point $K$ which coincides with the (left) internal focus $f_1'$ of the homothetic family. The polar image of the latter with respect to its right internal focus $f_2'$ is a mirrored, out-of-phase image of the left one. If $N$ is odd, the harmonic mean of the areas of the two shown lateral harmonic polygons (magenta) is experimentally invariant (\cref{conj:inv-area-sum}).}
    \label{fig:homot-polars}
\end{figure}

\subsection{Isocurves of Brocard angle}

Based on experimental evidence, in \cref{sec:isobrocs} we conjecture that a result by Johnson 
\cite{johnson17-schoute} for $N=3$ remains valid for all $N$. Namely, that the isocurves of inversion centers for constant Brocard angle are circles in a special pencil known as the Schoute pencil, containing the circumcircle and Brocard circle of the family (defined in \cref{app:harm}).


\subsection*{Related Work}

Original results concerning harmonic polygons can be found in \cite{casey1888,simmons1886,tarry1887}. In \cite{sharp45}, the harmonic family is defined as a generic projection of a regular polygon, but in this case metric properties are lost. In \cite[Section 4.6, p. 129]{akopyan12}, the harmonic porism is studied in the Klein model of hyperbolic plane (where $K$ becomes the center of the ideal circle). A recent study of harmonic quadrilaterals is  \cite{pamfilos2014-harmonics}.

The more elementary Brocard porism of triangles is studied in  \cite{bradley2011-brocard,bradley2007-brocard,johnson1960,simmons1888-recent}. In \cite{garcia2020-brocard-loci} loci of triangle centers over the Brocard porism is studied while \cite{reznik2022-brocard-converging} a converging sequence of such porisms is analyzed.

Quantities conserved by Poncelet $N$-gons have appeared in recent studies, including: (i) the confocal pair \cite{garcia2020-new-properties,reznik2020-eighty,reznik2021-fifty-invariants}, (ii) the homothetic pair \cite{galkin2021-affine}, (iii) the bicentric family \cite{roitman2021-bicentric}, and (iv) other special families \cite{bellio2021-parabola-inscribed,garcia2022-steiner-soddy}.

In \cite{bernhart59} a certain polynomial is proposed which suggests that a collection of expressions are invariant over the harmonic family (this is related to our results in \cref{sec:conserved}).
    


\subsection*{Article organization} In the next section we review the basics of the harmonic polygon family. Two new conserved quantities are proved in \cref{sec:conserved}; conservations based on the sum of powers of cotangents are proved in \cref{sec:cotangents}; the relationship between the harmonic family and Poncelet homothetics is derived in \cref{sec:homothetics}. A conjecture regarding isocurves of constant  Brocard angle appears in \cref{sec:isobrocs}. Videos of some experiments appear in \cref{sec:videos}.

In \cref{app:review} we review the basic construction and geometry of harmonic polygons. To facilitate further exploration, \cref{app:harm} provides explicit formulas for vertices and objects associated with the harmonic family.

%% file: 020_conserved.tex










In this section we will prove that some geometrical quantities are invariant for elements of the Ponceletian family of harmonic polygons. In the discussion that follows, we will identify the elements of $\mathbb{R}^2$ with the complex numbers. We will use the construction for a family of harmonic polygons $\P$  described in  \cite[Sec. VI, Prop. 2, p. 207]{casey1888} and shown in \cref{fig:pedro}(left)\footnote{This is identical to the one on \cref{fig:harm-proj}, where $d=S'$.}. Let $\alpha=\pi/N$.

\begin{itemize}
    \item Let $C$ be the unit circle centered at the origin, $d \in \mathbb{R}$ such that $|d|<1$ and $ \{ z_{k}=e^{i\left(2\alpha k+t \right)} \}$, $k=1,...,N$, $t\in \mathbb{R}$. For each $t\in \mathbb{R}$, the points $z_{k}$ are the vertices of a regular N-gon $R$ inscribed in $C$. The one-dimensional family of such regular polygons will be denoted by $\mathcal{R}$.

\item Consider the line through $d$ and $z_{k}$ and let $w_{k}$ be the other intersection of this line with $C$. For each $d$, the set of such points, in the natural order, are the vertices of a harmonic polygon $P$.
\end{itemize}

\begin{figure}
\begin{subfigure}[c]{0.48\textwidth}
    \includegraphics[trim=275 30 450 25,clip,width=\linewidth]{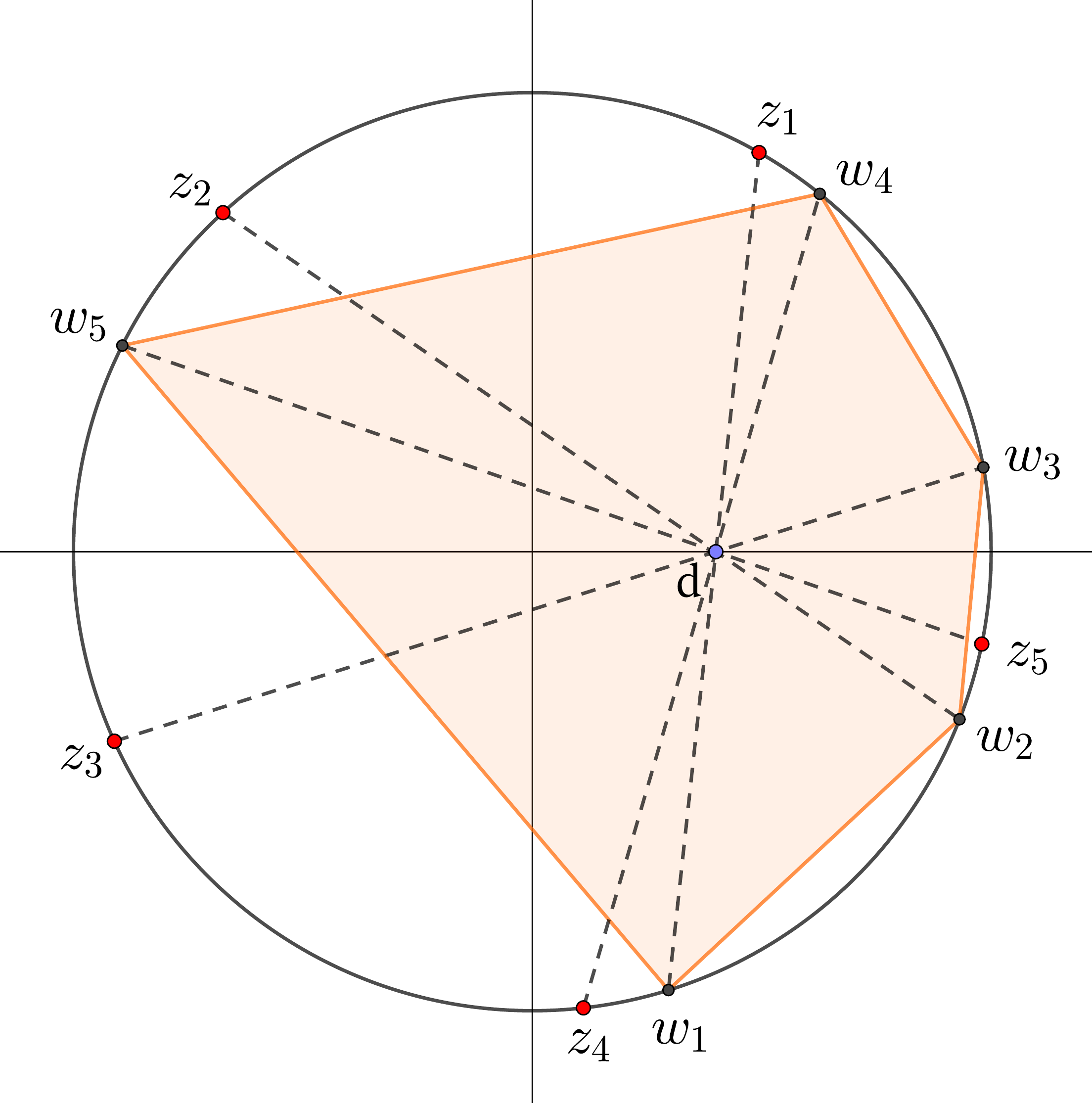}
\end{subfigure}
\begin{subfigure}[c]{0.51\textwidth}
    \includegraphics[trim=275 25 400 25,clip,width=\linewidth]{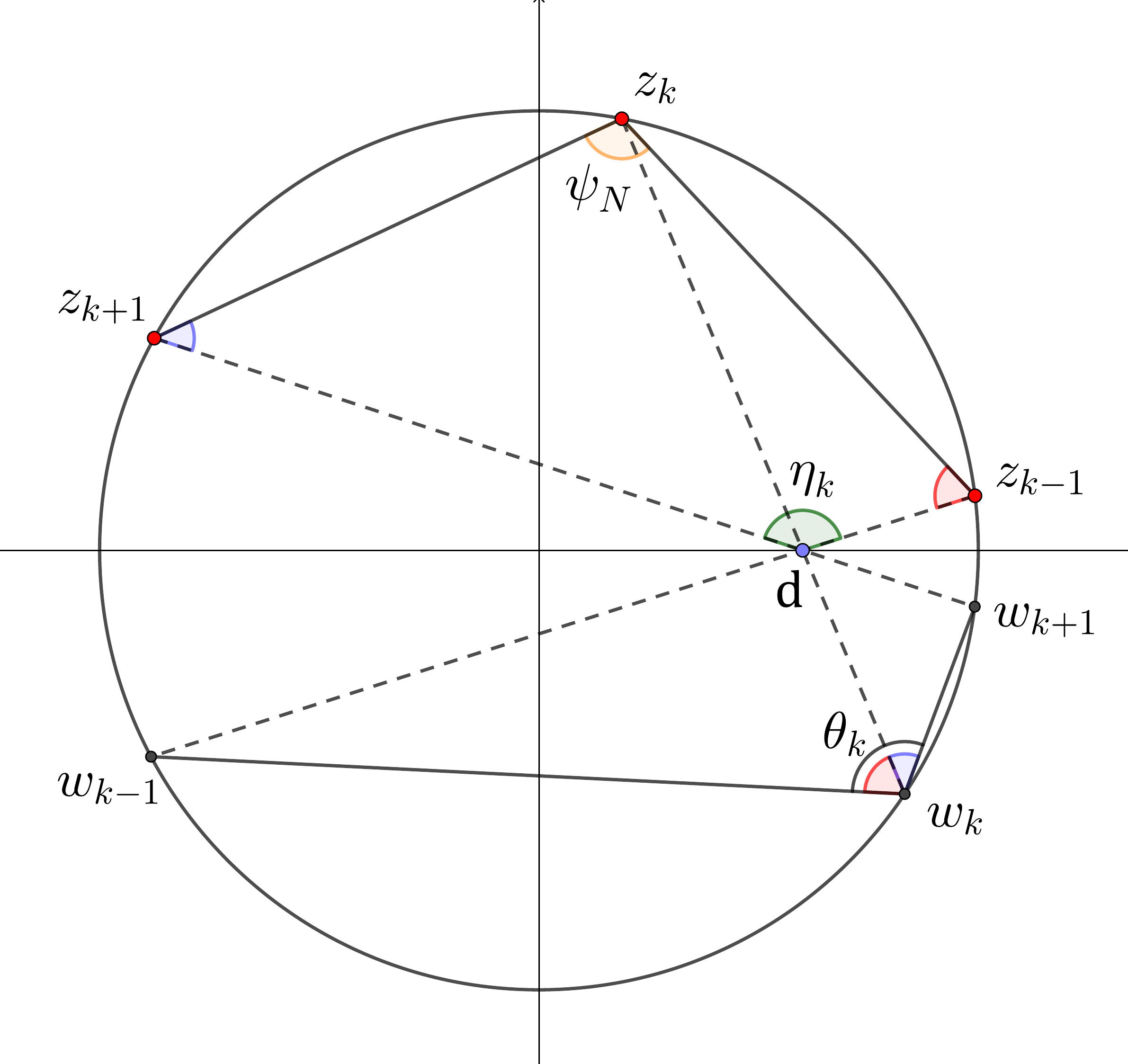}
\end{subfigure}
\caption{\textbf{Left:} The construction for a harmonic polygon used in the proofs in \cref{sec:conserved}. Note that this is equivalent to the construction in  \cref{fig:harm-proj} where the point $d$ above corresponds to $S'$. \textbf{Right:} Angle chasing used in the proof of \cref{lem:chasing}.}
\label{fig:pedro}
\end{figure}

\subsection{Inverse squared sidelengths}

Let $s_i$ denote the $i$-th sidelength of a harmonic polygon, $i=1,\ldots,N$.

\begin{proposition}
Over $\mathcal{P}$ the sum of inverse squared sidelenghts is invariant and given by:
\[\sum_{k=1}^N \frac{1}{s_k^2}= {N} \frac{ d^2 \cos^2\alpha+(d^{4}+1)/4 }{ (1-d^2)^2 \sin^2\alpha}
\]
\end{proposition}

\begin{proof}
By the geometric condition that defines a vertex $w_{k}$ of $P$ in terms of a vertex $z_{k}$ of $R$, we have:
\[w_{k}=\frac{d\bar{z}_{k}-1}{\bar{z}_{k}-d}\]
Using this expression for $w_{k}$ and the corresponding one for $w_{k-1}$, a simple computation yields:
\[w_{k}-w_{k-1}=\frac{(\bar{z}_{k}-\bar{z}_{k-1})(1-d^2)}{(\bar{z}_{k}-d)(\bar{z}_{k-1}-d)}\]

From the fact that $\left|\bar{z}_{k}-\bar{z}_{k-1}\right|=2\sin \alpha$, since it is the length of a side of $R$, we may conclude that:
\[s_{k}^{-2}=\left|w_{k}-w_{k-1}\right|^{-2}=\frac{\left| z_{k}-d \right|^{2}\left| z_{k-1}-d \right|^{2}}{4(1-d^2)^2\sin^2\alpha}\]

By the law of cosines, it follows that:

\begin{align*}
\left| z_{k}-d \right|^{2}&=1+d^2-2d\cos(\nu_{k}), \\
\left| z_{k-1}-d \right|^{2}&=1+d^2-2 d\cos(\nu_{k-1})
\end{align*}
where $\nu_{k}=2\alpha k+t$ and  $\nu_{k-1}=2\alpha (k-1)+t$. So:
{\small
\[
\left| z_{k}-d \right|^{2}\left| z_{k-1}-d \right|^{2}=(1+d^2)^2-2d(1+d^2)(\cos(\nu_{k})+\cos(\nu_{k-1}))+4d^2\cos(\nu_{k})\cos(\nu_{k-1})\]
}

When we sum over $k$, it is clear that the sum of $\cos(\nu_{k})$ and $\cos(\nu_{k-1})$ are both zero, so that the only non-trivial sum to evaluate is:
\[
\sum_{k=1}^N \cos(\nu_{k})\cos(\nu_{k-1})
\]

Since:
\[
\cos(\nu_{k})\cos(\nu_{k-1})=\frac{1}{2}\left(\cos(\nu_{k}+\nu_{k-1})+\cos(\nu_{k}-\nu_{k-1})\right)
\]

We may write the above sum as: 
\[
\frac{1}{2}\sum_{k=1}^N{\cos\left(2\alpha (2k-1)+2t\right)}+\frac{N}{2}\cos\left(2\alpha \right)
\]
It is well known that the above sum is equal to zero, see for example  \cite{knapp2009}. A short computation then yields the desired expression for $\sum_{k=1}^N (1/{s_k^2})$.
\end{proof}	

\subsection{Apollonius' radii}

\begin{definition}[Apollonius' Circles]
Given a triangle, one of the three circles passing through a vertex and both isodynamic points $S$ and $S'$ \cite[Isodynamic Points]{mw}.
\end{definition}

Referring to \cref{fig:apoll}, for each vertex $w_{k}$ in a harmonic polygon, consider the ``generalized'' Apollonius circle $C_{k}$ passing through the points $w_{k}$, $d$ and $d^{-1}$ (these are the limiting points of the generalized Schoute pencil \cite{johnson17-schoute}). Let $r_{k}$ be the radius of $C_{k}$. We will prove that:

\begin{figure}
    \centering
    \includegraphics[width=\textwidth]{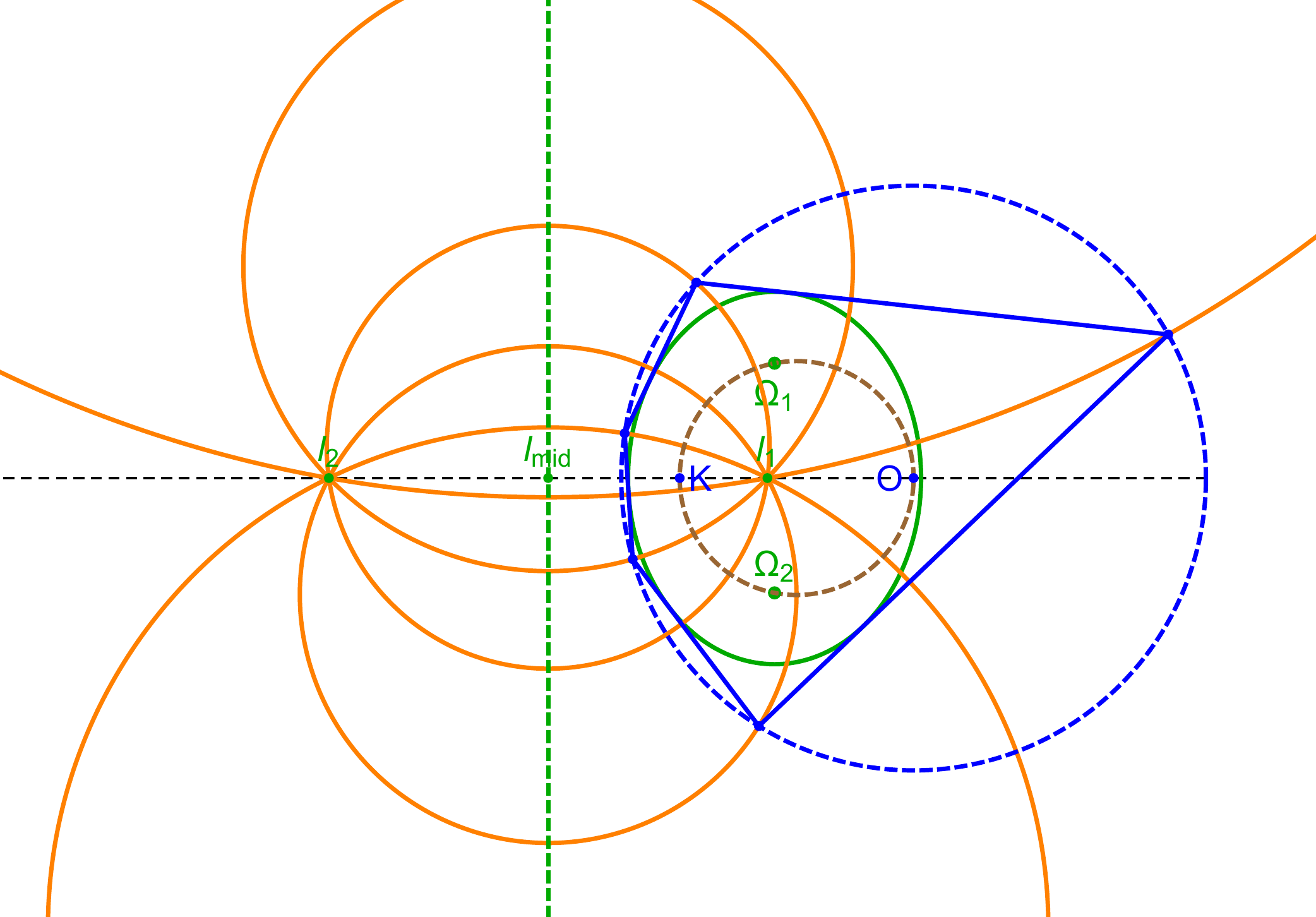}
    \caption{A harmonic polygon $\P$ (blue) and its Apollonius' circles (orange), each of which passes through one vertex of $\P$ and the two limiting points $\ell_1,\ell_2$ of the Schoute pencil. Also shown is the Lemoine axis (dashed green) of the pencil.}
    \label{fig:apoll}
\end{figure}

\begin{proposition}
Over $\mathcal{P}$, the sum of inverse squared Apollonius' radii is invariant and given by:
\[
\sum_{k=1}^N \frac{1}{r_k^2}=\frac{2N}{(d^{-1}-d)^2} 
\]
\end{proposition}

\begin{proof}
Let $\gamma_{k}=\angle d w_{k}d^{-1}$, then, by the law of sines, we have
\[
\frac{2\sin{\gamma_k}}{d^{-1}-d}=\frac{1}{r_k}
\]

A straightforward computation, using for instance the complex cross ratio, shows that the points $0,w_{k},d^{-1}$ and $z_{k}$ are concyclic, and from this we conclude that $\gamma_{k}=2\alpha k+t$ $mod$ $2\pi$.  
Therefore:
\[
\sum_{k=1}^{N}\frac{1}{r_{k}^2}=\frac{4}{(d^{-1}-d)^2}\sum_{k=1}^{N}\sin^2{\gamma_k}
\]  
Using the identity $\sin^2{x}=\left(1-\cos(2x)\right)/2$ and the fact that:
\[
\sum_{k=1}^{N}\cos{2\gamma_k}=0
\]
we conclude that
\[
\sum_{k=1}^{N}\sin^2{\gamma_k}=\frac{N}{2}
\]
and therefore:
\[
\sum_{k=1}^{N}\frac{1}{r_{k}^2}=\frac{2N}{(d^{-1}-d)^2}
\]
Which yields the claim.
\end{proof}

%% file: 025_cotangents.tex
The following lemma contains a useful expression for the cotangent of an internal angle of a harmonic polygon.  Henceforth, let $\rho=(1+d^2)/(d^2-1)$.


\begin{lemma}
Let $P$ be a harmonic polygon and $\theta_{k}$ be the internal angle of $H$ at the vertex $w_{k}$, then:
\begin{equation}
\label{cotint}
\cot{\theta_k}=\frac{-2d\cos{\left(2\alpha k+t\right)}}{(d^2-1)\sin(2\alpha)}+\rho\cot{\left(2\alpha \right)}
\end{equation}
\label{lem:chasing}

\end{lemma}

\begin{proof}


Referring to \cref{fig:pedro}(right), the internal angle $\psi_{k}$ of a regular $N$-gon at $z_{k}$ is fixed and given by $\psi_k=\psi_N=\frac{(N-2)\pi}{N}$;  $\eta_{k}$ is the angle $\angle{z_{k+1} d z_{k-1}}$, then, from elementary geometry, we have $\theta_{k}+\eta_{k}+\frac{(N-2)\pi}{N}=2\pi$ and therefore it follows that:
\[\cot{\theta_{k}}=-\cot{\left(\eta_{k}+\frac{(N-2)\pi}{N}\right)}=\frac{1+\cot{\eta_{k}}\cot{2\alpha }}{\cot{\eta_{k}}-\cot{2\alpha }} \]

We will first compute $\cot{\eta_{k}}$. Note that, if we denote by $\left\langle \, , \, \right\rangle$ the canonical inner product in $\mathbb{R}^2$, then:
\[ \cot{\eta_{k}}=\frac{\left\langle z_{k-1}-d ,z_{k+1}-d \right\rangle}{\left\langle i(z_{k-1}-d),z_{k+1}-d \right\rangle} \]

The numerator can be computed using complex multiplication as follows, first we write:
\[
\left\langle z_{k-1}-d ,z_{k+1}-d \right\rangle= \text{Re}\left[(z_{k-1}-d)(\bar{z}_{k+1}-d)\right]\]

\noindent Using the well-known trigonometric identity,
\[\cos{\varphi}+cos{\psi}=2\cos{\left(\frac{\varphi+\psi}{2}\right)}\cos{\left(\frac{\varphi-\psi}{2}\right)}\]

\noindent A straightforward computation yields:
\[\left\langle z_{k-1}-d ,z_{k+1}-d \right\rangle=-2d\cos{\left(2\alpha k+t\right)}\cos{\left(2\alpha \right)}+\cos(4\alpha)+d^2\]

\noindent Analogously, the denominator, which we will denote by $\Delta$, is given by:
\[\Delta=\left\langle i(z_{k-1}-d),z_{k+1}-d \right\rangle=-2d\cos{\left(2\alpha k+t\right)}\sin{\left(2\alpha \right)}+\sin{(4\alpha)}\]

With an explicit expression for $\cot{\eta_k}$, we can now compute  $\cot{\theta_k}$. To simplify the expressions, we will compute the numerator and denominator of $\cot{\theta_k}$ separately. Let's start with  the denominator $\mathcal{D}=\cot{\eta_{k}}-\cot{2\alpha }$:

\begin{align*}
\mathcal{D} = & \frac{1}{\Delta}\left(-2d\cos{(2\alpha k+t)}\cos(2\alpha)+\cos{(4\alpha )}+d^2 \right)-\\
    & \frac{\cot(2\alpha)}{\Delta}\left(\sin{(4\alpha )-2d\cos{(2\alpha k+t)}\sin(2\alpha)}\right) \\
  = &\frac{1}{\Delta}\left(\cos{(4\alpha )}-\cot(2\alpha)\sin{(4\alpha )}+d^2\right)\\
  = &\frac{1}{\Delta}\left(d^2-1\right) 
\end{align*}

Since the numerator $\mathcal{N}=1+\cot{(\eta_k)}\cot(2\alpha)$ can be computed in a similar way, we limit ourselves to write down the result:

\[
\mathcal{N}=\frac{1}{\Delta}\left( \frac{-2 d\cos{(2\alpha k+t)}}{\sin{(2\alpha )}}+\cot(2\alpha)(1+d^2)\right)\]

Thus, we have:
\[
\cot{\theta_k}=\frac{-2d\cos{(2\alpha k+t)}}{(d^2-1)\sin{(2\alpha )}}+\rho\cot(2\alpha)
\]
This concludes the proof.
\end{proof}

Using (\ref{cotint}), we may obtain explicit expressions for conserved quantities. As an example, we have the following proposition.

\begin{proposition}
Over $\mathcal{P}$, the sum of (i) cotangents and (ii) squared contangents of internal angles are invariant and given by:
\begin{align*}
\sum_{k=1}^N\cot\theta_k&=N\rho\cot(2\alpha)\\
\sum_{k=1}^N\cot^2\theta_k&=N \frac{  \rho^2(2+\cos(4\alpha))  - 1 }{ 1 - \cos(4\alpha)}
\end{align*}
\end{proposition}

\begin{proof}
From the \cref{lem:chasing}, it follows that
\begin{align*}
\sum_{k=1}^N\cot\theta_k&=\sum_{k=1}^N\left[\frac{-2d\cos{(2\alpha k+t)}}{(d^2-1)\sin{(2\alpha )}}+\rho\cot(2\alpha)\right]=N\rho\cot(2\alpha)\\
\sum_{k=1}^N\cot^2\theta_k&= \frac{4d^2}{(d^2-1)^2\sin^2{(2\alpha )}}\sum_{k=1}^N\cos^2(2\alpha k+t)\\
&-\sum_{k=1}^N\frac{4d\rho\cot(2\alpha)\cos{(2\alpha k+t)}}{(d^2-1)\sin{(2\alpha )}}+N\rho^2\cot^2(2\alpha)
\end{align*}
\noindent Using the  following known identity \cite{knapp2009}:
\[
\sum_{k=1}^N\cos^2(2\alpha k+t)=\frac{N}{2},
\]
obtain:
\begin{align*}
\sum_{k=1}^N\cot^2\theta_k&=\frac{N}{(1-d^2)^2}\left[\frac{2d^2}{\sin^2{(2\alpha )}}+\cot^2(2\alpha)(1+d^2)^2\right]=\\
&=N\left[ \frac{2d^2}{(1-d^2)^2\sin^2(2\alpha)} + \rho^2 \cot^2(2\alpha)\right]
\end{align*}
A simple computation, using trigonometric identities, yields the desired expression for the above sum and concludes the proof.
\end{proof}

\subsection{Symmetric invariants}
To discuss a set of invariant quantities involving the elementary symmetric functions of the cotangents of the internal angles of harmonic polygons, we will use the following notation for such functions:

Let $X=(X_{1},X_{2},...,X_{N})$ and let $e_{k}(X)$ denote the elementary symmetric functions in the variables $X_{j}$ ($j=1,\ldots,N$) that is, $e_{0}(X)=1$, $e_{1}(X)=\sum_{j=1}^{N} X_{j}$, $e_{2}(X)=\sum_{1\leq j<i\leq N} X_{i}X_{j},\ldots,e_{N}(X)=X_{1}X_{2} \ldots X_{N}$.

Our next result is a generalization of the invariance of the sum of cotangents of the internal angles $\theta_i$ of harmonic polygons.  

\begin{theorem}
Let $P$ be a harmonic $N$ sided polygon, $\lambda=(\cot{\theta_1},\cot{\theta_2,...,\cot{\theta_N})}$, then, the polynomials $e_{1}(\lambda), e_{2}(\lambda),...,e_{N-1}(\lambda)$ are invariant, that is, they do not depend on $t$.
\label{thm:symmetric}
\end{theorem}

\begin{proof}
By the \cref{lem:chasing},  $e_{k}(\lambda)$ is a linear combination (with constant coefficients) of the elementary symmetric functions of the variables  
$$c_{j}=\cos{(2\alpha j+t)},$$
for $j=0,...,k$. Therefore, it suffices to prove that $e_{1}(c), e_{2}(c),...,e_{N-1}(c)$, where 
$$c=(c_{1},c_{2},...,c_{N}),$$
are invariant. Since $e_{k}(c)$ is a sum of products of cosines, then, the trigonometric identity
\[
 \prod_{i=1}^{m}\cos{\theta_{i}}=\frac{1}{2^{m-1}}\sum_{p_{n}\in p}{\cos{(p_n)}}
\]
where $p$ is the set of $2^{m-1}$ numbers having the form $\theta_{1} \pm \theta_{2} \pm...\pm \theta_{m}$,  
allows one to express $e_{k}(c)$ as a linear combination of cosines. The general term of this combination has the form
\[
A_m\cos {(mt+\varphi_{m})}
\]
where $m$ varies from $0$ to $k$, and $A_m$ and $\varphi_{m}$ are constants. This general term can be rewritten as: 
\[
a_m\cos {(mt)}+b_m\sin {(mt)}
\]
Except for $m=0$, such terms are periodic functions with period $2\pi/{m}$.

But notice that $e_{k}(c)$ is a periodic function with period $2\alpha $, with $N>k$. Therefore, from the well known orthogonality of trigonometric functions, it follows that $a_{m}=b_{m}=0$ for all $m\neq 0$. In other words, $e_{k}(c)$ must be constant.
\end{proof}

\subsection{Higher cotangent powers}

As shown in \cref{tab:cot-higher}, the sum of cotangents of powers $k$ higher than $2$ will also be invariant, when $N>k$. This can be regarded as a corollary to \cref{thm:symmetric}.

\begin{table}
\begin{tabular}{|c|cccccc|}
\hline
k & N=3 & N=4 & N=5 & N=6 & N=7 & N=8 \\
\hline
1 & $\checkmark$ & 0 & $\checkmark$ & $\checkmark$ & $\checkmark$ & $\checkmark$ \\
2 & $\checkmark$ & $\checkmark$ & $\checkmark$ & $\checkmark$ & $\checkmark$ & $\checkmark$ \\
3 &  & 0 & $\checkmark$ & $\checkmark$ & $\checkmark$ & $\checkmark$ \\
4 &  &  & $\checkmark$ & $\checkmark$ & $\checkmark$ & $\checkmark$ \\
5 &  & 0 &  & $\checkmark$ & $\checkmark$ & $\checkmark$ \\
6 &  &  &  &  & $\checkmark$ & $\checkmark$ \\
7 &  & 0 &  &  &  & $\checkmark$ \\
\hline
\end{tabular}
\caption{A $\checkmark$ (resp. $0$) indicates that for a given $N$, $\sum\cot^k(\theta_i)$ is invariant. Note that we get invariance if (i) $N>k$ and (ii) $N=4$, odd $k$, in which case the sum is zero.}
\label{tab:cot-higher}
\end{table}
Since for $N=4$ opposite angles are supplementary:

\begin{corollary}
If $N=4$, $\sum\cot^k(\theta_i)=0$ for all odd $k$.
\end{corollary}
  


  



\subsection{Comparing conservations across Poncelet families}

\cref{tab:conserve} shows Conservations proved side-by-side with those manifested by other Poncelet families, described and/or proved in \cite{akopyan2020-invariants,bialy2020-invariants,caliz2020-area-product,galkin2021-affine,reznik2021-fifty-invariants,roitman2021-bicentric}.

\begin{table}
\setlength{\tabcolsep}{2pt}
{\small
\begin{tabular}{|c||c|c|c|c|c|}
\hline
invariant & Confocal & Bicentric & Inversive & Homothetic & Harmonic \\
\hline
$L$ & \checkmark$^o$ & & \checkmark \cite{roitman2021-bicentric} & & \\
\hline
$A$ & & & & \checkmark$^o$ & \\
\hline
$L/A$ & & \checkmark$^o$ &  & & \\
\hline
$\sum{s_i}^2$ & & & & \checkmark \cite{galkin2021-affine} & \\
\hline
$\sum{s_i}^2/A$ & & & & \checkmark \cite{galkin2021-affine} & \checkmark$^o$ \\
\hline
$\sum{s_i^{-2}}$ & & & & & \checkmark \\
\hline
$\sum{r_i^{-2}}$ & & & & & \checkmark \\
\hline
$\sum{\cos}$ & \checkmark \cite{akopyan2020-invariants,bialy2020-invariants,caliz2020-area-product} & \checkmark \cite{roitman2021-bicentric} & \checkmark$^\dagger$ \cite{roitman2021-bicentric}  & & \\
\hline
$\sum{\cot}$ & & & & \checkmark \cite{galkin2021-affine} & \checkmark \\
\hline
$\sum{\cot^2}$ & & & & \checkmark$^\dagger$ \cite{galkin2021-affine} & \checkmark \\
\hline
$\sum{(\sin \cos)}/L$ & & $\checkmark$ & & & \\
\hline
$\sum{(\sin \cos)}/A$ & & $\checkmark$ & & & $\checkmark$ \\
\hline
$A_1 A_2$ & & \checkmark$^*$ & \checkmark$^*$ & & \\
\hline
{\scriptsize $A_1^{-1}+A_2^{-1}$} & & & & & \checkmark$^*$ \\
\hline
\hline
polar of & Bicentric & Confocal & -- & Harmonic & Homothetic \\
\hline
{\scriptsize
\makecell[cc]{Inversion\\Center}} & $\ell_1$ & $f_1,f_2$ & -- & $K$ & $f_1',f_2'$ \\
\hline
\end{tabular}
}
\caption{Quantities conserved by various Poncelet families and their polar-derived families. An $^o$ after a $\checkmark$ indicates the quantity is well-known. References are provided to extant proofs. The last two lines (Polar) indicate how to obtain the current family as the polar image of some other family with respect to a circle centered on the indicated inversion center. For the case of side areas ($A_1,A_2$) both foci are needed. Notes: $\dagger$: $N{\neq}4$. $*$: odd $N$.}
\label{tab:conserve}
\end{table}

%% file: 040_homothetics.tex
In this section we derive the  transformations required to jump from one of regular, harmonic, homothetic, to another. Let $\R$, $\P$, $x_0$, and $\alpha$ be as in the previous section. We omit most proofs since they were obtained with the aid of a Computer Algebra System (CAS). 

\subsection{From harmonics to homothetics}

Referring to \cref{fig:homot-polars}:

\begin{proposition}
The polar image of $\P$ with respect to a unit circle centered on the symmedian point $K$ of $\P$ is a new Poncelet family $\H$ of polygons interscribed between two homothetic, concentric ellipses $\E_H$ (external) and $\E_h$ (internal) given by: 
\begin{align*}
\E_{H}:& \frac{(x-x_H)^2}{a_H^2}+\frac{y^2}{b_H^2}-1=0,\;\;\;\E_h:   \frac{(x-x_h)^2}{a_h^2}+\frac{y^2}{b_h^2}-1=0,
\\
a_h& =\frac{(x_0^2 + 1)^2}{|1-x_0^2| },\;\; b_h=x_0^2+1,\;\; x_h=\frac{x_0(3x_0^4 + 3x_0^2 + 2)}{x_0^4-1}\\
a_H &=a_h/\cos\alpha,\;\; b_H   =b_h/\cos\alpha,\;\; x_H=x_h
\end{align*}
\end{proposition}

\subsection{From homothetics back to harmonics} Let $\H$ be a family of Poncelet $N$-gons interscribed between two concentric, homothetic ellipses $\E_H=(a_H,b_H)$ and $\E_h=(a_h,b_h)$ with common centers at $(0,0)$. Let $f_h=(-c_h,0)$ be a focus of $\E_h$, where $c_h^2=a_h^2-b_h^2$. 

\begin{proposition}
The polar image of $\H$ 
with respect to a unit circle centered on $f_h$ is a harmonic family inscribed in a circle $\C_1=(O_1,R_1)$ and circumscribing an ellipse $\E_1$ with semiaxes $(a_1,b_1)$ and centered on $(x_1,0)$ where:

\begin{align*}
O_1=&\left[-c_h\frac{(1+b_h^2)}{b_h^2},0\right],\;\;\;R_1=\frac{a_h}{b_h^2}\\
x_1&=-c_h \left(a_h^2+(1 -c_h^{2}) \cos^2\alpha \right)/k^2\\
a_1&=  a_h\cos\alpha /k^2,\;\;\;b_1= {a_h\cos\alpha}/(b_h k)
\end{align*}
where $k^2= a_h^2-c_h^{2} \cos^2\alpha$. Furthermore, the symmedian $K_1$ of the harmonic family coincides with $f_h$.
\end{proposition}

\begin{corollary}
Let $\delta=|K_1-O_1|$.
\[ \left(\frac{\delta}{R_1}\right)^2 = 1-\left(\frac{b_h}{a_h}\right)^2 \]
\end{corollary}

\subsection*{Lateral harmonic areas}

Let $\H$ be a Poncelet family of $N$-gons interscribed between two homothetic, concentric ellipses $\E_H,\E_h$. Let $f_{h,1},f_{h,2}$ denote the foci of $\E_h$.
Let $A_1$ (resp. $A_2$) denote the area of the harmonic polygon which is a polar image of $\H$ with respect to a circle centered on $f_{h,1}$ (resp. $f_{h,2}$). Note that if $N$ is even, a polygon in the homothetic family is centrally symmetric. Therefore, $A_1=A_2$, with each area variable. When $N$ is odd, these areas are in general distinct.

\begin{proposition}
For $N=3$ and $N=5$, $1/A_1+1/A_2$ is invariant and given by:

\begin{align*}
N=3:&\;\; \frac{\sqrt{3}}{18}\frac{b}{a} (a^2+3b^2)\\
N=5:& \;\;\frac{b}{40\sin(2\pi/5)a}
{\frac { \left( {a}^{4}+10\,{b}^{2}{a}^{2}+5\,{b}^{4} \right)  
 \left(  \sqrt {5}(a^2+3\,{b}^{2}) +5\,{a}^{2}+7\,{b}^{2}
 \right) }{    5\,{a}^{4}+10\,{b
}^{2}{a}^{2}+{b}^{4}  }}\\
\end{align*}
\end{proposition}

Experimentally, the following holds:

\begin{conjecture}
For any odd $N$, $1/A_1+1/A_2$ is invariant.
\label{conj:inv-area-sum}
\end{conjecture}

If the \cref{conj:inv-area-sum} holds then:
\begin{corollary}
$\frac{1}{\sum{s_{i,1}^2}}+\frac{1}{\sum{s_{i,2}^2}}$ is invariant.
\end{corollary}

This stems from the fact that for any harmonic polygon $\cot\omega=\sum{s_i^2}/(4A)$ \cite[§16, pp. 298]{simmons1886}, where $s_i$ is the ith sidelength of a harmonic polygon, and both polar images (by symmetry of the foci with respect to the center of the homothetic family) have the same $\omega$.

\begin{conjecture}
$\frac{\sum{\sin(2\theta_i)}}{A}$ is invariant. Equivalently, $\frac{\sum{\sin(2\theta_i)}}{\sum{s_i^2}}$ is invariant. 
\end{conjecture}


If the \cref{conj:inv-area-sum} holds then:

\begin{corollary}
$\frac{1}{\sum{\sin(2\theta_{i,1})}}+\frac{1}{\sum{\sin(2\theta_{i,2})}}$ is invariant.
\end{corollary}

\subsection{Closing the loop}

Let $\R$, $\P$, $\H$ be as above. Referring to \cref{fig:3-fams}, below we specify transformations which interchange families in the triad. Below let $\P(N,\omega)$ denote a family of harmonic $N$-gons with Brocard angle $\omega$.

\begin{figure}
    \centering
    \includegraphics[trim=75 0 125 0,clip,width=\textwidth,frame]{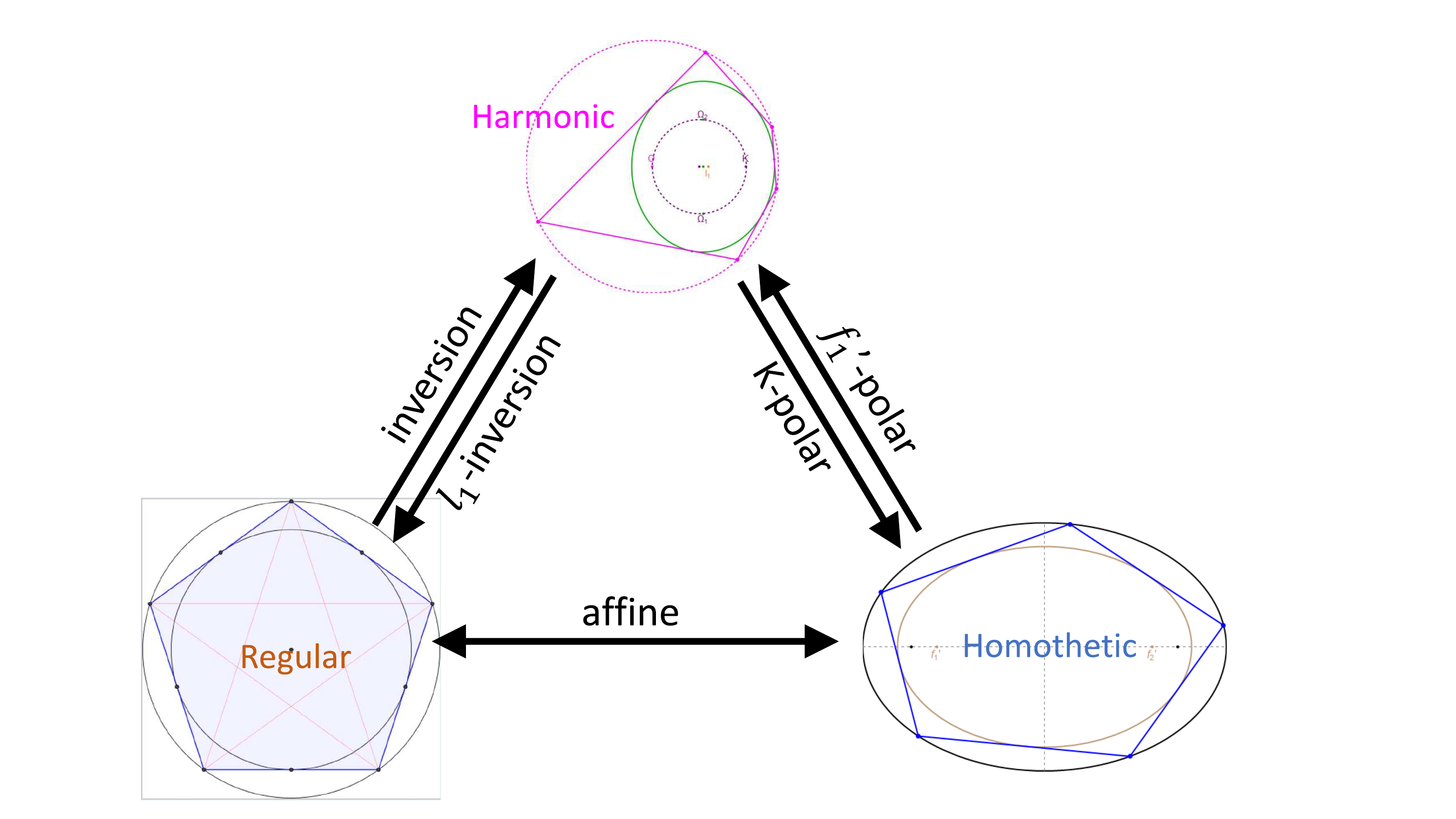}
    \caption{The three families mentioned in this article  are inversive, affine, or polar images of each other. Note that the inversive relationship between regular and harmonic is equivalent to the projection of \cref{fig:harm-proj}.}
    \label{fig:3-fams}
\end{figure}

\begin{proposition}
The inversive image of $\R$ with respect to a unit circle centered on $(x_0,0)$ is $P(N,\omega)$ if $x_0= \sqrt{\frac{1-\tan\alpha \tan\omega}{1+\tan\alpha \tan\omega}}$.
\end{proposition}

Let $\H$ be a family of $N$-gons which is an affine image of the $\R$, where $(x,y) \to (k x, y)$. Clearly, $\H$ is bounded by two homothetic, concentric ellipses $\E,\E'$ where $\E=(k,1)$ and, using the geometry of regular polygons, $\E'=(k \cos\alpha,\cos\alpha)$.

\begin{proposition}
The polar image of $\H$ with respect to a focus of $\E'$ will be $\P(N,\omega)$ if $k = \cot\alpha \cot\omega$.
\end{proposition}

Let $a,b$ denote the semiaxes of the the inner ellipse in a Poncelet homothetic family $\H$.

\begin{proposition}
The polar image of $\H$ with respect to an internal focus will be identical to the inversive image of $\R$ with respect to a unit circle centered on $(x_0,0)$ if
$x_0 = \sqrt{\frac{a+b}{a-b}}$.
\end{proposition}

%% file: 050_isobrocs.tex
Let $P$ be a polygon in a Poncelet family of harmonic $N$-gons, and $P'$ be another $N$-gon whose vertices are inversions of those of $\P$ with respect to a circle centered at some point $Q$. Recall the Schoute pencil $\S$ of a harmonic polygon is the one containing both circumcircle and the Brocard circle. Johnson 
\cite{johnson17-schoute} shows that for the $N=3$ case, the locus of $Q$ such that the Brocard angle of $P'$ is constant are individual circles in $\S$. Referring to \cref{fig:harm-pencil}, sufficient experimental evidence suggests:

\begin{conjecture}
The locus of $Q$ such that the Brocard angle of $P'$ is constant are individual circles in $\S$. Furthermore, if $Q$ is on the Lemoine axis or Brocard circle, the Brocard angles of both $P$ and $P'$ are equal.
\label{conj:brocard}
\end{conjecture}

\begin{figure}
    \centering
    \includegraphics[width=\textwidth]{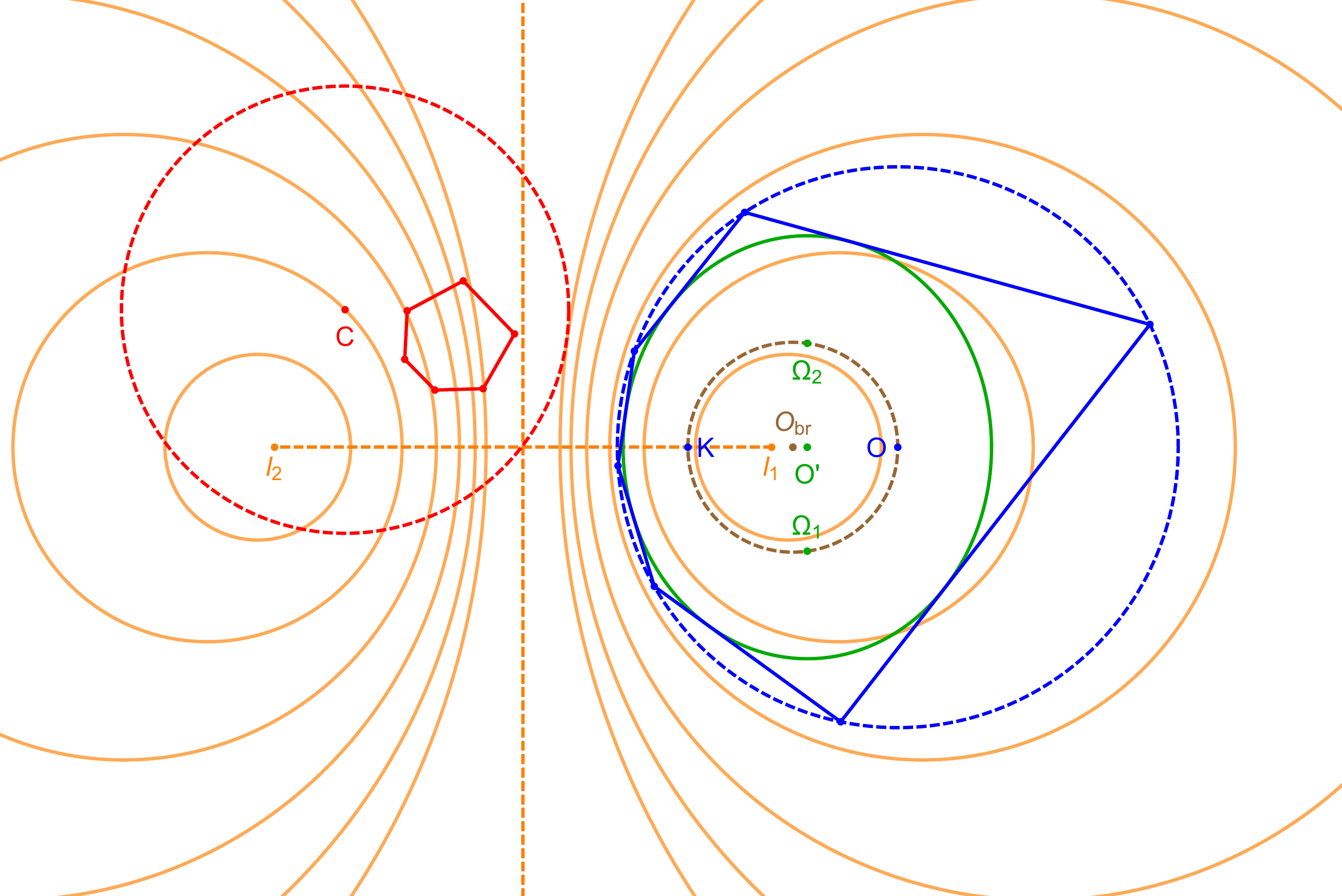}
    \caption{\textbf{Isocurves of Brocard angle}: a harmonic polygon (blue) is shown as well as a few circles (orange) in its Schoute pencil, containing the circumcircle (dashed blue) and the Brocard circle (dashed brown). \cref{conj:brocard} states that these circles are isocurves for centers of inversion $C$ such that the Brocard angle $\omega'$ of the $C$-inversive polygon (red) is constant. If $C$ is on the Brocard circle or Lemoine line (vertical dashed orange), $\omega'$ is equal to the Brocard angle $\omega$ of the reference harmonic polygon.}
    \label{fig:harm-pencil}
\end{figure}

%% file: 090_videos.tex
Animations illustrating some invariant phenomena herein are listed on Table~\ref{tab:playlist}.

\begin{table}[H]
\small
\begin{tabular}{|c|c|l|l|}
\hline
id & N & Title & \textbf{youtu.be/<.>}\\
\hline
01 & 5 & {Invariants of $\F$} &
\href{https://youtu.be/2PdsC3CcqaE}{\texttt{2PdsC3CcqaE}}\\
02 & 3 & {Invariant Brocard Angles over $F$ 3-gons} & \href{https://youtu.be/2fvGd8wioZY}{\texttt{2fvGd8wioZY}} \\
03 & 3 & {Locus of Brocard Points $\F$ 3-gons} & \href{https://youtu.be/13i3JGY-fK4}{\texttt{13i3JGY-fK4}}\\
04 & 5 & {Invariant signed area of Evolute Polygon, $s=\{.25,.5,.75,.1\}$} & \href{https://youtu.be/JCj0q7_hlA8}{\texttt{JCj0q7\_hlA8}} \\
05 & 3,5,6,8 & {Evolute Polygons with Zero Signed Area} & \href{https://youtu.be/3nvXYFoI5Wg}{\texttt{3nvXYFoI5Wg}} \\
06 & 5 & {Invariant-Area Evolute Polygon with $s=1$} & \href{https://youtu.be/ChsfLzKrb4o}{\texttt{ChsfLzKrb4o}} \\
07 & 3 & {Zero-area Evolute Polygon is a horizontal segment} & \href{https://youtu.be/f80QaYs5_J4}{\texttt{f80QaYs5\_J4}} \\
08 & 3 & {Two zero-area evolute polygons intersect on $X_{76}$} & \href{https://youtu.be/OFA_j25R8ks}{\texttt{OFA\_j25R8ks}} \\
\hline
\end{tabular}
\caption{Illustrative videos. The last column is clickable and provides the YouTube code.}
\label{tab:playlist}
\end{table}

%% file: 120_ack.tex
\noindent We would like to thank A. Akopyan for valuable discussions, and the anonymous referee for a meticulous review and corrections. The first author is fellow of CNPq and coordinator of Project PRONEX/CNPq/FAPEG 2017 10 26 7000 508.

%% file: 230_app_review_harmonic.tex
As shown in \cref{fig:harm-proj}, a harmonic polygon is the projective image of vertices of a regular $N$-gon; specifically, that corresponding vertices are harmonic conjugates with respect to a projective center $S$ and an axis $\sigma$ \cite{simmons1886,tarry1887}.

\begin{figure}
    \centering
    \includegraphics[width=\textwidth]{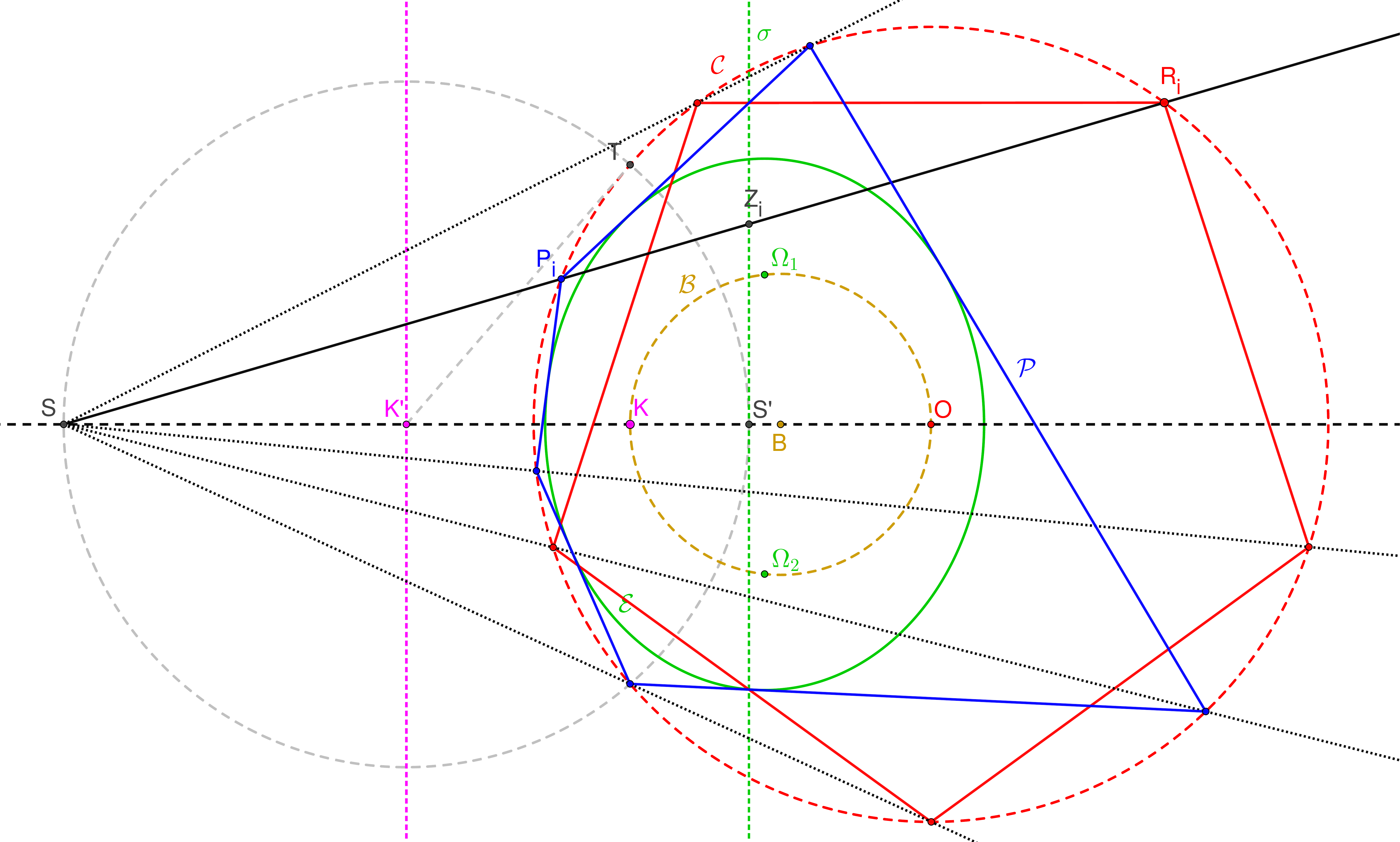}
    \caption{\textbf{Harmonics via projection}: Let $\R$ be a regular polygon (red) and $\C$ its circumcircle (dashed red) centered on $O$. A harmonic polygon $\P$ (blue) can be obtained as follows \cite{simmons1886,tarry1887}: (i) choose a point $K$ -- the symmedian point -- in the interior of $\C$; (ii) $K'$ is the intersection of $OK$ with the polar of $K$ (dashed magenta) with respect to $\C$; (iv) let $T$ be a tangent to $\C$ from $K'$; (v) let $S,S'$ be points on $OK$ which lie on the circle (dashed gray) centered on $K'$ and passing through $T$; (vi) for every regular vertex $R_i$, $P_i$ is at the intersection of $S R_i$ with $\C$. As it turns out, $(P_i,R_i)$ are harmonic conjugates with respect to $S$ and $Z_i$, where $Z_i$ is the intersection of $S R_i$ with the ``projection axis'' $\sigma$ through $S'$ and perpendicular to $OK$. Also shown in the Brocard inellipse (green) $\E$ which $\P$ circumscribes. Its foci $\Omega_1$ and $\Omega_2$ are known as the Brocard points. Also shown is the Brocard circle $\B$  (dashed brown) passing $K$, $O$, $\Omega_1$, and $\Omega_2$. Since $\P$ is interscribed in $\C,\E$, it triggers a (Poncelet) porism over which the Brocard points and Brocard circle remain stationary. Note: it can be shown $S$ and $S'$ are the limiting points of the pencil containing $\C$ and $\B$ and the polar of $K$ wrt $\C$ is the Lemoine axis of $\P$.}
    \label{fig:harm-proj}
\end{figure}

In another construction, it is regarded as the inversive image of vertices of a regular polygon with respect to a chosen inversion center \cite{casey1888}, see \cref{fig:reg-inv}. In yet another construction, it is simply a generic projection of any Ponceletian family \cite{sharp45}, though in this case is not interested in Euclidean properties specific to the case where the outer conic is a circle.

\begin{figure}
    \centering
    \includegraphics[width=.8\textwidth]{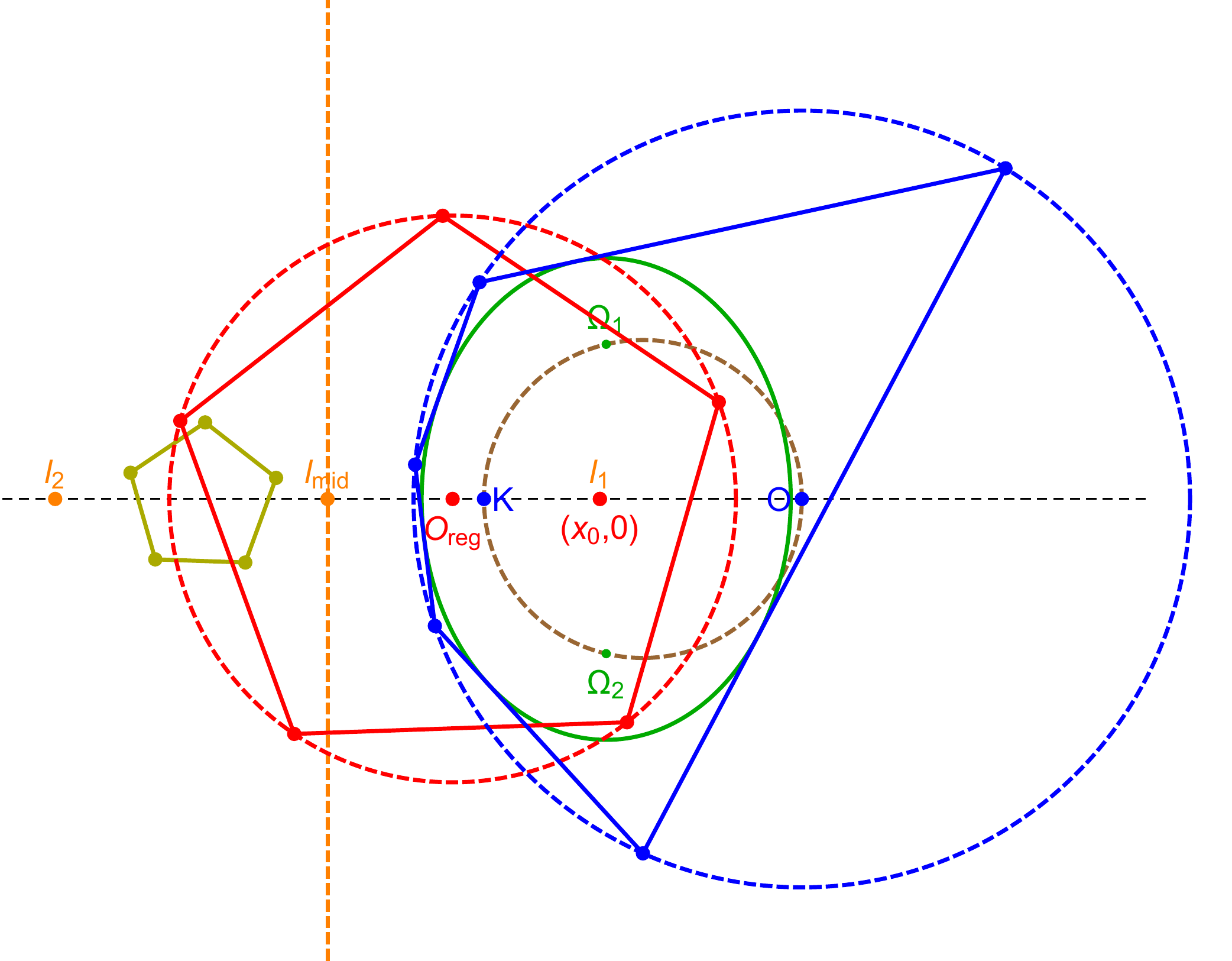}
    \caption{\textbf{Harmonics via inversion}: The vertices of a a harmonic polygon (blue) are the inversions of those of a regular polygon (red) with respect to a center of inversion $(x_0,0)$. The latter coincides with a limiting point $\ell_1$ of the pencil $\Pi$ of circles containing the circumcircle (dashed blue) and Brocard circle (dashed brown) of the resulting harmonic polygon. Conversely, the inverting the vertices of a harmonic polygon with respect to either limiting point $\ell_1$ or $\ell_2$ of $\Pi$ produces two distinct regular polygons (red and olive). Also shown is the Lemoine axis (vertical dashed orange) of the harmonic family, which cuts $OK$ at the midpoint of $\ell_1$ and $\ell_2$ and can be regarded as the infinite-radius circle in $\Pi$.}
    \label{fig:reg-inv}
\end{figure}

An equivalent, though not constructive definition, is that a polygon is harmonic if an interior point $K$ can be located such that its distance to each side is at a fixed proportion to each sidelength \cite{casey1888}. $K$ is called the {\em symmedian point}.  

While not all polygons are harmonic, all triangles are, since a symmedian point always exists\footnote{Its trilinears -- which are proportional to the distance to each side -- are, as expected, the sidelengths.}, denoted $X_6$ in  \cite{etc}. Referring to  \cref{fig:harm-n3}, the so-called ``Brocard porism'' is one of triangles interscribed between their circumcircle and fixed Brocard inellipse (whose foci are the stationary Brocard points of the family). Simmons calls these ``co-brocardal'', since all Brocard geometry objects (Brocard points, Brocard circle, Lemoine axis, etc.) remain stationary, see \cite{bradley2007-brocard,simmons1888-recent}.

\begin{figure}
    \centering
    \includegraphics[width=.8\textwidth]{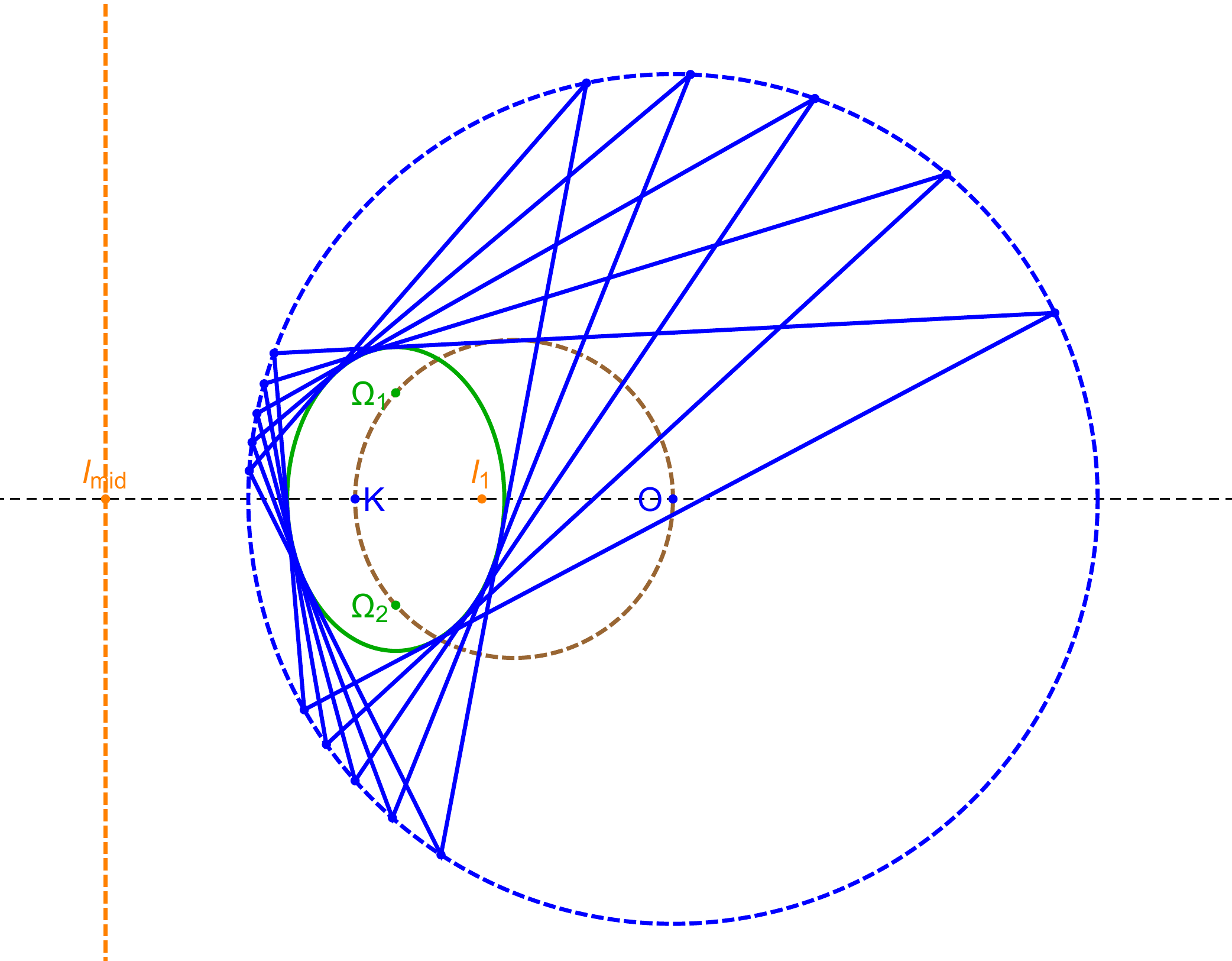}
    \caption{\textbf{The Brocard porism (all triangles are harmonic)}: A porism of triangles interscribed between their circumcircle (dashed blue) and Brocard inellipse (green). Using Kimberling's notation, $O$ and $K$ are the circumcenter $X_3$ and the symmedian point $X_6$, respectively. The pencil containing the circumcircle (dashed blue) and Brocard circle (brown) is known as the Schoute pencil \cite{johnson17-schoute}. Its limiting points $\ell_1$ and $\ell_2$ (not shown) are the two isodynamic point $X_{15}$ and $X_{16}$. Their midpoint $\ell_\text{mid}$ is $X_{187}$ on \cite{etc}.}
    \label{fig:harm-n3}
\end{figure}

Referring to \cref{fig:basic}, for any $N$, a porism of harmonic $N$-gons conserves a key quantity known as the {\em Brocard angle} $\omega$ defined as follows: an angle such that a counterclockwise (resp. clockwise) rotation of all sides $P_i P_{i+1}$ about $P_i$ will pass through $\Omega_1$ (resp. $\Omega_2$). A key identity, valid for all harmonic polygons is \cite{casey1888}:
\[ \cot\omega =  \frac{\sum{s_i^2}}{A} \]
where $s_i$ are the sidelengths and $A$ is the area, variable over the porism. I.e., this suggests that (i) the sum of internal angle cotangents and (ii) the ratio of squared sidelengths by area are conserved. Note that for $N=3$, $\cot\omega=\sum\cot\theta_i$ \cite[Brocard angle]{mw}.

%% file: 220_app_harmonic.tex
Consider the family of regular $N$-gons $\R$ centered on the origin and inscribed in a unit circle. Let $\P$ denote the harmonic polygon which is the inversive image of $\R$ with respect to a unit circle centered on  $C_0=[x_0,0]$. Let $\alpha=\pi/N$. The following expressions refer to objects associated with $\P$:

\subsection{Harmonic Vertices}
\begin{align*}
    [x_i,y_i]&=\left[-\frac{      \left( 1-
2\,x_0^{2} \right) \cos( 2\alpha i+t) +  x_0^{3}    }  {  2\,x_0\, \cos(2\alpha i+t)   -1-x_0^{2}   },  
  -\frac{\sin(2\alpha i+t)}{2x_0\cos(2\alpha i+t) - 1-x_0^2 }\right]
\end{align*}

\subsection{Circumcircle \torp{$\C=[O,R]$}{C}}

\[O=\left[{\frac { \left( x_0^{2}-2 \right) x_0}{ x_0^{2}-1}},0\right],\;\;\;R= \frac{1}{|x_0^{2}-1|} \]

\subsection{Brocard points \torp{$\Omega_1$,$\Omega_2$}{omega1,omega2}}

\[ \Omega_{1,2}= \frac{1}{k}\left[ 
   (2 x_0^2 - 1)\cos(2\alpha) -  x_0^4 +  x_0^2 - 1, \pm \sin(2\alpha)\right] \]
  where $k=2  x_0 \cos(2\alpha) -  x_0^3 - 1/x_0$.

\subsection{Brocard inellipse \torp{$\E$}{E}}

\begin{align*} \E:\;& \frac{(x-x_c)^2}{a^2}+\frac{y^2}{b^2}=1\\
a&= \frac{(1-x_0^2)\cos\alpha}{k'},\;\;\;b=\frac{\cos\alpha}{ \sqrt{k'}}
\end{align*}
where $x_c$ is the x-coordinate of $\Omega_1$ and $k'= (x_0^2 + 1)^2-(2 x_0 \cos\alpha)^2$. The eccentricity $\varepsilon$ of $\E$ is given by:
\[\varepsilon=\frac{c}{b}=
\frac{2|x_0|\sin\alpha}{\sqrt{(x_0^2 + 1)^2-4{x_0^2}\cos^2{\alpha}}}
\]

\subsection{Symmedian point \torp{$K$}{K}}
\[K= \left[\frac{x_0^3}{ x_0^2 + 1}, 0\right]\]

Let $\delta=|K-O|$. It can be shown that:
\[
x_0 =\frac{  1\pm\sqrt{ 1-(\delta/R)^2}     }{  (\delta/R)}
\]

Note that the product of the two possible $x_0$ is unity.
 
\subsection{Brocard circle  \torp{$\C'=[O',r]$}{C'}}

\[ O'=\left[\frac{x_0 (x_0^4 - x_0^2 - 1 )}{x_0^4 - 1},0\right],\;\;\;r =  \left|\frac{x_0}{x_0^4 - 1}\right| \]



\subsection{Limiting points \torp{$\ell_{1,2}$}{l1,2} of \torp{$\C$}{C} and \torp{$\C'$}{C'}}
\[\ell_1=[x_0,0],\;\; \ell_2=\left[\frac{x_0^2 - 1}{x_0},0\right]
\]


\subsection{Brocard angle \torp{$\omega$}{w}}

Casey gives the relation \cite[Prop. 3, pp. 209]{casey1888}:
\[\tan\omega =\sqrt{1-(\delta/R)^2}\cot\alpha\]
where $\delta = |K-O| = 2 r$. This can also be expressed as:
\[\tan\omega =
\frac{ |1-x_0^2|}{ 1+x_0^2 } \cot\alpha
\]
This implies that: 
\[ x_0 = \pm \sqrt{\frac{\cot\alpha-\tan\omega}{\cot\alpha+\tan\omega}} \]






